\documentclass[a4paper]{amsproc}
\usepackage{amssymb}
\usepackage{amscd}
\usepackage[dvips]{graphicx}
\usepackage[all,cmtip]{xy}
\usepackage[inline]{enumitem}
\usepackage{xcolor}
\usepackage{cite}
\usepackage[hidelinks]{hyperref}
\usepackage{comment}

\theoremstyle{plain}
 \newtheorem{thm}{Theorem}[section]
 \newtheorem{prop}{Proposition}[section]
 \newtheorem{lem}{Lemma}[section]
 
\theoremstyle{definition}
 \newtheorem{exm}{Example}[section]
 \newtheorem{rem}{Remark}[section]
\newtheorem{dfn}{Definition}[section]

\numberwithin{equation}{section}

\allowdisplaybreaks

\newcommand{\R}{\mathbb{R}}

\newcommand{\ddim}{\mathrm{ddim\;}}
\newcommand{\dind}{\mathrm{dind\;}}
\newcommand{\corank}{\mathrm{corank\;}}
\newcommand{\grad}{\mathrm{grad\;}}

\DeclareMathOperator{\Span}{\mathrm{span\,}}

\DeclareMathOperator{\Id}{\mathrm{Id}}

\renewcommand{\le}{\leqslant}

\title[INTEGRABLE SYSTEMS IN COSYMPLECTIC GEOMETRY]{INTEGRABLE SYSTEMS IN COSYMPLECTIC GEOMETRY}

\author{Bo\v zidar Jovanovi\'c, Katarina Luki\'c}

\address{B.J.: Mathematical Institute of the  Serbian Academy of Sciences and
Arts, Kneza Mihaila 36, 11000 Belgrade, Serbia}
\email{bozaj@mi.sanu.ac.rs}

\address{K.L.:  Faculty of Mathematics, University of Belgrade, Studentski trg 16, 11000 Belgrade, Serbia}
\email{katarina\_lukic@matf.bg.ac.rs}



\begin{document}

\begin{abstract}
Motivated by the time-dependent Hamiltonian dynamics,  
we extend the notion of Arnold-Liouville and noncommutative integrability of Hamiltonian systems on symplectic manifolds to that on cosymplectic manifolds. We prove a variant of the non-commutative integrability for evaluation and Reeb vector fields on cosymplectic manifolds and provide a construction of cosymplectic action-angle variables.
\end{abstract}

\maketitle


\section{Introduction}

Recall  that the equation
\begin{equation}
\dot x=X
 \label{eq}
\end{equation}
on an $n$--dimensional manifold $M$ is
\emph{(non-Hamiltonian) completely integrable} if there is an open dense
subset $M_{reg}\subset M$ with a proper submersion
\begin{equation}\label{sub*}
\pi=(f_1,\dots,f_p): M_{reg}\longrightarrow W\subset \R^p
\end{equation}
and an Abelian Lie algebra
$\mathcal X$ of symmetries containing $X$ such that:
\begin{itemize}
\item[(i)] the vector field $X$ is tangent to the fibers
of $\pi$, i.e., $f_i$ are the first integrals of the flow;

\item[(ii)] the fibers of $\pi$ are orbits of $\mathcal
X$.
\end{itemize}

If \eqref{eq} is completely integrable then $M_{reg}$ is foliated
on $(n-p)$--dimensional tori with linear dynamics that can be determined by quadratures.
Alternatively, the condition that the map is proper can be replaced by the condition that the vector
fields in $\mathcal X$ are complete. Then the flow is linear over invariant manifolds that are diffeomorphic to
$\mathbb T^l\times \mathbb R^{n-p-l}$, where $\mathbb T^l$ is an $l$--dimensional torus (see \cite{Ar})).
Note that if we relax the condition that the algebra $\mathcal X$ is Abelian by the condition that it is solvable, we
still have solvability of the equation by quadratures \cite{Koz, CFG}.

However, the above definition does not reflect
a possible underlying geometrical structure of the equations.
The concept of complete integrability within symplectic, Poisson (Arnold--Liouville  integrability),
and contact geometry is very well studied (e.g., see \cite{Ar, BJ, LMV, MF, N, BM, KT, Jo, JJ1}). There is also a notion of integrability on almost-symplectic and
$b$-symplectic manifolds \cite{FS, KM}.

Recently, a general setup for integrability that include  symplectic, contact and Dirac structures is performed by Zung \cite{Zung}.
In this paper, motivated by time-dependent Hamiltonian dynamics, we consider some specific aspects of the integrability within the framework of cosymplectic
geometry, which are not covered in \cite{Zung}.

In Section \ref{sec2} we recall on the basic definitions in cosymplectic geometry and prove a variant of the non-commutative integrability of evaluation vector fields within cosymplectic setting  (the Nekhoroshev type formulation in Theorem \ref{TEOREMA} and the Mishchenko--Fomenko type formulation in Theorem \ref{TEOREMA2}). 
In Section \ref{sec3} we construct  the cosymplectic analogue of the
 Nekhoroshev action--angle variables \cite{N} (see Theorem \ref{TEOREMA3}), or, in the commutative case,
the cosymplectic analogue  of the standard action--angle variables given by  the  Arnold--Liouville theorem    \cite{Ar} (see Subsection \ref{sec3.2}).

The time $t$ in time-dependent Hamiltonian systems is a globally defined Casimir function of the associated Poisson structure. We consider a general case, where a globally defined Casimir function  does not need to exist. It is interesting that the Hamiltonian flows of the first integrals are quasi-periodic, but not integrable in the sense of the usual definition of
 complete integrability of Hamiltonian flows on Poisson manifolds (see  Subsection \ref{sec3.2}).

Note that  the Arnold--Liouville  integrability in the time-dependent case is studied in \cite{GMS} within a framework of symplectic principal $\R$--bundles (e.g.,
see \cite{LMP, GS}).
The integrability of time-dependent Hamiltonian systems can be considered also within the framework of contact geometry
(e.g., see \cite{BM, KT, Jo, Jo1}).

\section{Integrability on cosymplectic manifolds}\label{sec2}

\subsection{Notation}
Recall that a
cosymplectic manifold $(M,\omega,\eta)$ is a  $(2n+1)$--dimensional
manifold $M$ endowed with a closed 2--form $\omega$
and a closed 1--form $\eta$, such that $\eta\wedge\omega^n$ is a volume form  (see \cite{Al, CLL}).
The \emph{gradient vector
field} $\grad f$ of the function $f\in C^\infty(M)$ is defined by
the conditions
$$
i_{\grad f}\omega=df-Z(f)\eta, \qquad \eta(\grad f)=Z(f),
$$
where $Z$ is the \emph{Reeb vector field} determined by
$$
i_Z\omega=0, \qquad \eta(Z)=1.
$$

The cosymplectic manifold is a corank 1 Poisson manifold, with the Poisson bracket defined by
\begin{equation}\label{PB}
\{f,g\}=\omega(\grad f,\grad g).
\end{equation}
Its symplectic leaves are leaves of the integrable distribution $\ker\eta$.

Besides the gradient vector field, to every smooth function $f\in C^\infty(M)$, one can associate the \emph{Hamiltonian vector field} $X_f$ of the Poisson
structure
\eqref{PB}
$$
X_f(g)=\{g,f\}, \qquad g\in C^\infty(M),
$$
and the \emph{evaluation vector field}
$$
Y_f=Z+X_f.
$$

The Hamiltonian and evaluation vector fields can also be defined by relations
\begin{equation}\label{definicije}
i_{X_f}\omega =df-Z(f)\eta, \quad \eta(X_f)=0, \qquad i_{Y_f}\omega=df-Z(f)\eta, \quad \eta(Y_f)=1.
\end{equation}

We have the following important identities (we follow the sign conventions from \cite{Al, CLL}).
\begin{eqnarray}
 && i_{X_f}\omega=i_{\grad f}\omega,\label{r1} \\
&& \{f,g\}=\omega(X_f,X_g),\label{r2}\\
 && X_{\{f,g\}}=-[X_f,X_g],\label{r3}\\
 && [Z,X_f]=X_{Z(f)}.\label{r4}
\end{eqnarray}

Any point has a neighborhood with canonical coordinates $(t,q_1,\dots,q_n,p_1,\dots,p_n)$, such that
$$
\omega=dq \wedge dp=dq_1\wedge dp_1+\dots+dq_n\wedge dp_n, \qquad \eta=dt, \qquad Z=\frac{\partial}{\partial t}.
$$

In canonical coordinates we find
\begin{align}
\grad f &=\frac{\partial f}{\partial t}\frac{\partial}{\partial t}+\sum_i \frac{\partial f}{\partial p_i}\frac{\partial}{\partial q_i}-\frac{\partial f}{\partial
q_i}\frac{\partial}{\partial p_i},\\
X_f &=\sum_i \frac{\partial f}{\partial p_i}\frac{\partial}{\partial q_i}-\frac{\partial f}{\partial q_i}\frac{\partial}{\partial p_i},\\
Y_f &=\frac{\partial}{\partial t}+\sum_i \frac{\partial f}{\partial p_i}\frac{\partial}{\partial q_i}-\frac{\partial f}{\partial q_i}\frac{\partial}{\partial p_i},
\end{align}

The cosymplectic manifolds provide a natural framework for the time-dependent Hamiltonian mechanics
 (e.g., see \cite{EG, LS}). Namely, consider a non-autonomous Hamiltonian equations
\begin{equation} \label{1}
\frac{dq_i}{dt}=\frac{\partial H}{\partial p_i},\qquad
\frac{dp_i}{dt}=-\frac{\partial H}{\partial q_i}, \qquad
i=1,\dots,n,
\end{equation}
in the canonical coordinates $(q_1,\dots,q_n,p_1,\dots,p_n)$ of the cotangent bundle $T^*Q$ of the configuration space $Q$.
It is natural to consider the equation \eqref{1}
in the extended phase space $\R\times T^*Q(t,q,p)$ by adding the equation $\dot t=1$.
Then the Hamiltonian system becomes an autonomous one, of the form  $\dot x=Y_H$, where $Y_H$ is the
evaluation vector field
\begin{equation}\label{reb1}
Y_H=\frac{\partial}{\partial t}+\sum_i{\frac{\partial H}{\partial p_i}\frac{\partial}{\partial q_i}-
\frac{\partial H}{\partial q_i}\frac{\partial}{\partial p_i}}
\end{equation}
on the cosymplectic manifold $(\R\times T^*Q, dq\wedge dp,dt)$.
Alternatively, we can consider
the \emph{Poincar\'e--Cartan} 1--form
\begin{equation}\label{PCF}
\alpha=pdq-Hdt=p_1dq_1+\dots+p_ndq_n-Hdt
\end{equation}
and the cosymplectic manifold
\begin{equation}
\label{csm}(\R\times T^*Q, -d\alpha=dq\wedge dp +dH\wedge dt,dt).
\end{equation}
Then the vector field \eqref{reb1} coincides with the Reeb vector field $Z$ determined by the conditions
$i_Z (d\alpha)=0$ and $dt(Z)=1$ (e.g., see \cite[Chapter 9]{Ar}).

\subsection{Linearization}
Consider the flow of the evaluation vector field $Y_H$:
\begin{equation}
\dot x=Y_H\label{EVF}
\end{equation}
on a $(2n+1)$-dimensional cosymplectic manifold $(M,\omega,\eta)$.
The function $f$ is the first integral of \eqref{EVF} if
\begin{equation}\label{int}
Y_H(f)=Z(f)+X_H(f)=Z(f)+\{f,H\}=0.
\end{equation}

First, we formulate an analogue of Nekhoroshev theorem on non-commutative integrability (see \cite{N}).

\begin{thm}\label{TEOREMA} Assume that the system \eqref{EVF} has $m$ independent integrals
\[
\pi=(f_1,\dots,f_m): \, M  \longrightarrow \R^m,
\]
such that $f_1,\dots,f_r$ Poisson commute with all integrals
\begin{equation}\label{PC}
\{f_i,f_j\}=0, \quad i=1,\dots,r, \quad j=1,\dots,m,
\end{equation}
 and $2n=m+r$. Then:

{\rm (i)} The equations \eqref{EVF} are
locally solvable by quadratures on an open subset $M_{reg}$ where $d\pi(x)$ is of the maximal rank:
$\mathcal X=\langle Y_H, X_{f_1},\dots,X_{f_r}\rangle$ is an Abelian Lie algebra of symmetries, which has the
fibers of $\pi$ as orbits.

{\rm (ii)} If the vector fields $Y_H,X_{f_1},\dots,X_{f_r}$ are
complete on $M_{reg}$, then a connected component $M_\mathbf c^0$ of the invariant
variety
\begin{equation}\label{inv povrs2}
M_\mathbf c=\pi^{-1}(\mathbf c)\cap M_{reg}
\end{equation}
is diffeomorphic to $\mathbb T^l \times \R^{r+1-l}$, for some $l$, $0\le l \le r+1$. There exist coordinates $\varphi_1,\dots,\varphi_l,x_1,\dots,x_{r+1-l}$
of $\mathbb T^l \times \R^{r+1-l}$, which linearise the equation \eqref{EVF}:
\begin{eqnarray*}
&&\dot\varphi_i=\omega_i=\mathrm{const}, \qquad i=1,\dots,l,\\
&&\dot x_j=a_j=\mathrm{const}, \qquad  j=1,\dots,r+1-l.
\end{eqnarray*}
\end{thm}

In the case where $m=r=n$, we have a variant of the Arnold--Liouville integrability on cosymplectic manifolds.

\begin{proof}
(i)
First, note that the vector fields $Y_H,X_{f_1},\dots,X_{f_r}$ commute between themselves. The relations
$$
[X_{f_i},X_{f_j}]=0, \qquad i,j=1,\dots, r
$$
follow directly from \eqref{r3} and \eqref{PC}. Further, from \eqref{r4} we have
\begin{align*}
[Y_H,X_{f_i}]&=[X_H+Z,X_{f_i}]=-X_{\{H,f_i\}}+X_{Z(f_i)}\\
&=X_{X_H(f_i)}+X_{Z(f_i)}=X_{Y_H(f_i)}=0, \qquad\quad i=1,\dots,r.
\end{align*}
Next, since $f_1,\dots,f_m$ are integrals of \eqref{EVF}, $Y_H$ is tangent to $M_\mathbf c$, while from $X_{f_i}(f_j)=\{f_j,f_i\}=0$, $i=1,\dots,r$, we have that
$X_{f_1},\dots,X_{f_r}$ are tangent to $M_\mathbf c$ as well.

It only remains to prove that $Y_H,X_{f_1},\dots,X_{f_r}$ are independent vector fields. Since
$X_{f_1},\dots,X_{f_r}$ are tangent to the symplectic leaves of the Poisson bracket, and $Y_H$ is transversal to symplectic leaves, we directly have that $Y_H$
is independent of the Hamiltonian vector fields $X_{f_1},\dots,X_{f_r}$.
Assume that there exist real numbers $a_1,\dots,a_r$ such that
\begin{equation}\label{s0}
a_1X_{f_1}+\dots+a_r X_{f_r}=0, \qquad a_1^2+\dots+a_r^2\ne 0,
\end{equation}
at some point $x_0\in M$. Then, from the definition
\eqref{definicije}, we get
\begin{equation}\label{s1}
(a_1df_1+\dots+a_r df_r)\vert_{x_0}=Z(a_1f_1+\dots+a_r f_r)\eta\vert_{x_0}.
\end{equation}
On the other hand, since $f_i$ are integrals, from \eqref{int} we have
\begin{align}
\nonumber Z(a_1f_1&+\dots+a_r f_r)+\{a_1 f_1+\dots+a_r f_r,H\}=\\
&Z(a_1f_1+\dots+a_r f_r)-(a_1X_{f_1}+\dots+a_r
X_{f_r})(H)=0.\label{s2}
\end{align}

By combining \eqref{s0}, \eqref{s1}, and \eqref{s2}, it follows
$$
(a_1df_1+\dots+a_r df_r)\vert_{x_0}=0.
$$

Thus, at regular points of the $(r+1)$--dimensional invariant level sets \eqref{inv povrs2}, the commuting vector fields
$Y_H,X_{f_1},\dots,X_{f_r}$ are independent and, by the Lie theorem \cite{Koz}, the equation \eqref{EVF} can be solved locally by quadratures.

\medskip

(ii) The proof of item (ii) is the same as the corresponding statement in the Arnold--Liouville theorem (see \cite{Ar}).
\end{proof}

As we already mentioned, the Hamiltonian equations with the
time-dependent Hamiltonian $H$  can be seen  as
the flow of the evaluation vector field \eqref{reb1} or the
the Reeb flow on the cosymplectic manifold
\eqref{csm}. Similarly, as it follows
from the following statement, the notion of integrability of the systems defined by the evaluation and the Reeb vector fields are naturally related.

\begin{prop}
Consider cosymplectic manifolds $(M,\omega,\eta)$ and $(M,\omega+dH\wedge\eta,\eta)$, where $H$ is a smooth function on $M$. By $Z,X_f,Y_f,\{\cdot,\cdot\}$
and  $Z',X'_f,Y'_f,\{\cdot,\cdot\}'$ denote the Reeb vector field, the Hamiltonian and  evaluation vector field of the function $f$, and the Poisson bracket on
$(M,\omega,\eta)$ and $(M,\omega+dH\wedge\eta,\eta)$, respectively.

{\rm (i)}  The Reeb vector field $Z'$  coincides with the
evaluation vector field $Y_H$.

{\rm (ii)} The Poisson brackets coincide as well: $\{f_1,f_2\}=\{f_1,f_2\}'$.
\end{prop}

\begin{proof} (i)
The Reeb vector field $Z'$ is determined by
$$
i_{Z'}(\omega+dH\wedge \eta)=i_{Z'}\omega+dH(Z')\eta-\eta(Z')dH=0, \quad \eta(Z')=1,
$$
which is equivalent to the definition of $Y_H$.

\medskip

(ii) It is enough to prove that the Hamiltonian vector fields $X_f$ and $X'_f$ coincide. The Hamiltonian vector fields $X_f, X'_f$ are defined by
\begin{align*}
& i_{X_f}\omega=df-Z(f)\eta, \\
& i_{X'_f}(\omega+dH\wedge \eta)=i_{X'_f}\omega+dH(X'_f)\eta-dH\eta(X'_f) =df-Z'(f)\eta, \\
&  \eta(X_f)=\eta(X'_f)=0.
\end{align*}

Since
$$
Z(f)=Y_H(f)-X_H(f)=Y_H(f)+\{H,f\}=Y_H(f)+X_f(H)
$$
and $Z'(f)=Y_H(f)$, we get, respectively,
\begin{align*}
& i_{X_f}\omega=df-dH(X_f)\eta-Y_H(f)\eta, \qquad \eta(X_f)=0,\\
& i_{X'_f}\omega=df-dH(X'_f)\eta-Y_H(f)\eta, \qquad \eta(X'_f)=0,
\end{align*}
implying $X_f=X'_f$.

Therefore, $\{f_1,f_2\}=X_{f_2}(f_1)=X'_{f_2}(f_1)=\{f_1,f_2\}'$.
\end{proof}

\begin{exm}{\rm
As in the case of integrable contact systems, where K-contact manifolds provide natural examples of integrable Reeb flows \cite{JJ1},
K-cosymplectic manifolds (see \cite{BG}) provide natural examples of cosymplectic integrable Reeb flows.
}\end{exm}

\begin{exm}
The examples with commutative symmetries ($m=r=n$) for time-dependent Hamiltonain systems on cosymplectic manifolds
\eqref{csm} can be found in \cite{LS}.
\end{exm}

\begin{rem}{\rm
The vector field \eqref{reb1} can be seen also as a section of the characteristic line bundle of the Poincar\'e--Cartan 1--form \eqref{PCF}. The integrability of the
vector field $Z$ by using the Noether symmetries of characteristic line bundles is considered in \cite{Jo1}. In the case
the Poincar\'e--Cartan 1--form is contact, one can also use the notion of contact integrability \cite{BM, KT, Jo, JJ1}.
}\end{rem}

\begin{rem}
If  $(M, \Sigma)$ is a $b$--symplectic manifold, then there is a natural cosymplectic structure on $\Sigma$, such that the Reeb vector field $Z$
of this cosymplectic structure is the modular vector field of $(M,\Sigma)$ restricted to $\Sigma$ (see \cite{Sh}).
It would be interesting to  relate Theorem \ref{TEOREMA} to the integrability concepts on $b$--symplectic manifolds given in \cite{KM}.
\end{rem}

\begin{rem} 
The notion of integrability on contact, almost-symplectic, $b$-symplectic and Dirac manifolds is essentially a variant
of the Arnold--Liouville integrability. On the other hand, we note that the problem of integrability in non-holonomic mechanics has many interesting and unexpected
geometrical aspects (see \cite{BMT}), for example invariant manifolds do not need to be tori (see \cite{FJ}).
Further, while for non-holonomic Chaplygin systems with an invariant measure
integrability is usually related to the Arnold--Liouville integrability by means of the Chaplygin
Hamiltonisation via time reparametrisation (e.g., see \cite{BMT, Na} and references therein),
the examples of integrable Chaplygin systems
with an invariant measure that do not allow Chaplygin Hamiltonisation were recently obtained in \cite{Jo2}.
\end{rem}

\subsection{Complete sets of integrals on cosymplectic manifolds}
Similarly like for the Hamiltonian systems on Poisson manifolds, we have:

\begin{lem}
Assume that $f_1$ and $f_2$ are integrals of the system \eqref{EVF}. Then their Poisson bracket $\{f_1,f_2\}$ is also the first integral
of the flow.
\end{lem}

\begin{proof}
Since $f_1$ and $f_2$ are integrals we have
\begin{equation}\label{f1f2}
Z(f_1)+\{f_1,H\}=0, \quad Z(f_2)+\{f_2,H\}=0.
\end{equation}
On the other side we have the Jacobi identity
\[
\{\{f_1,f_2\},H\}+\{\{H,f_1\},f_2\}+\{\{f_2,H\},f_1\}=0,
\]
which together with \eqref{f1f2} implies
\[
\{Z(f_1),f_2\}+\{f_1,Z(f_2)\}+\{\{f_1,f_2\},H\}=0.
\]

The proof follows from the identity
\[
Z(\{f_1,f_2\})=\{Z(f_1),f_2\}+\{f_1,Z(f_2)\},
\]
that can be easily proved in the canonical coordinates $(t,q_1,\dots,q_n,p_1,\dots,p_n)$.
\end{proof}

Therefore, we can always assume that we are dealing with sets of integrals that are closed under the Poisson bracket and
the condition \eqref{PC} can be replaced by the  analogue condition in the Mishchenko--Fomenko theorem on non-commutative integrability \cite{MF}.

\begin{dfn}
Consider a Poisson subalgebra $(\mathcal F,\{\cdot,\cdot\})\subset (C^\infty(M),\{\cdot,\cdot\})$. Let
\[
F_x=\Span\{df(x)\,\vert\,f\in\mathcal F\}\subset T_x^*M.
\]

We say that $\mathcal F$ is \emph{complete} on a cosymplectic manifold $(M,\omega,\eta)$ if there are $m$ functions $f_1,\dots,f_m\in\mathcal F$, such that
\begin{equation}\label{CI}
F_x=\Span\{df_i(x)\,\vert\,i=1,\dots,m\}, \qquad  df_1\wedge\dots\wedge df_m\ne 0
 \end{equation}
 on a open dense set $M_{reg}$,
where their Poisson brackets are of the form
\begin{equation}\label{CP}
\{f_i,f_j\}=a_{ij}(f_1,\dots,f_m),
\end{equation}
and
\begin{equation}\label{CC}
\dim \ker\{\cdot,\cdot\}\vert_{F_x}=r,  \qquad  m+r=2n.
\end{equation}
The number $m$ and $r$ are called the \emph{differential dimension} and the \emph{differential index} of $\mathcal F$ and they are denoted
by $\ddim\mathcal F$ and $\dind\mathcal F$, respectively.
\end{dfn}

Note that the completeness condition \eqref{CC},
\[
\ddim\mathcal F+\dind\mathcal F=\dim M-1,
\]
is weaker then the completeness condition for the corresponding Hamiltonain systems on $(M,\{\cdot,\cdot\})$:
$\ddim\mathcal F+\dind\mathcal F=\dim M+\corank\{\cdot,\cdot\}=\dim M+1$.

\begin{thm}\label{TEOREMA2} Assume that the system \eqref{EVF} has a complete algebra of integrals $\mathcal F$ and consider the momentum mapping
\begin{equation}\label{MMP}
\pi=(f_1,\dots,f_m): \, M_{reg}  \longrightarrow \R^m,
\end{equation}
where $f_1,\dots,f_m\in\mathcal F$ satisfy the completeness conditions \eqref{CI}, \eqref{CP} and \eqref{CC}. Then:

{\rm (i)} The equations \eqref{EVF} are
locally solvable by quadratures on $M_{reg}$.

{\rm (ii)} A connected compact component $M_\mathbf c^0$ of the invariant variety $M_\mathbf c=\pi^{-1}(\mathbf c)$
is diffeomorphic to a torus $\mathbb T^{r+1}$. There exist coordinates $\varphi_1,\dots,\varphi_{r+1}$
of $\mathbb T^{r+1}$, which linearise the equation \eqref{EVF}:
\begin{eqnarray*}
\dot\varphi_i=\omega_i=\mathrm{const}, \qquad i=1,\dots,r+1.
\end{eqnarray*}
\end{thm}

\begin{proof} The momentum mapping \eqref{MMP} is a Posisson mapping with the constant corank $r$ Poisson bracket on $D^m=\pi(M_{reg})$ defined by
$\{y_i,y_j\}_{D^m}=a_{ij}(y_1,\dots,y_m)$.
There exist an open neighborhood $B^m\subset D^m$ of
$\mathbf c=\pi(M^0_\mathbf c)$ with Casimir functions $G_1,\dots,G_r$ of the bracket $\{\cdot,\cdot\}_{D^m}$ (e.g., see \cite{LM}).
Then
\[
g_1=G_1(f_1,\dots,f_m),\dots, g_r=G_r(f_1,\dots,f_m)
\]
and independent functions defined on $U=\pi^{-1}(B^m)$
that commute with all integrals
\[
\{g_k,f_i\}=\{G_k,y_i\}_{D^m}=0.
\]

The rest of the proof is the same as above with the
Abelian Lie algebra of symmetries
$\mathcal X=\Span\{Y_H, X_{g_1},\dots,X_{g_r}\}$ which are tangent to the fibers of $\pi$ within $U$.
\end{proof}

\section{Cosymplectic action-angle coordinates}\label{sec3}

\subsection{Action-angle coordinates for the Reeb flow}
Assume that the Reeb flow
\begin{equation}\label{REEBflow}
\dot x=Z
\end{equation}
on a $(2n+1)$--dimensional cosymplectic manifold $(M,\omega,\eta)$ has
a complete set of integrals $\mathcal F$ and that generic regular invariant manifolds are compact.
According to Theorem \ref{TEOREMA2}, this is equivalent to the assumption of Theorem \ref{TEOREMA} with compact regular invariant manifolds.
We have $m$ independent integrals $f_1,\dots,f_m$,
such that $f_1,\dots,f_r$ Poisson commute with all integrals
\begin{equation}\label{PC*}
\{f_i,f_j\}=0, \quad i=1,\dots,r, \quad j=1,\dots,m,
\end{equation}
 and $2n=m+r$. Then $\mathcal X=\langle Z, X_{f_1},\dots,X_{f_r}\rangle$ is an Abelian Lie algebra of symmetries, which has the
fibers of $\pi$ as orbits.
A connected component $M_\mathbf c^0$ of the invariant variety
\begin{equation}\label{inv povrs}
M_\mathbf c=\pi^{-1}(\mathbf c)\cap M_{reg}
\end{equation}
is diffeomorphic to a torus $\mathbb T^{r+1}$.

Now, as a modification of the construction of generalized action-angle coordinates in the
non-commutative version of the Arnold--Liouville theorem on symplectic manifolds (see \cite{N}), we get:

\begin{thm}\label{TEOREMA3}
There exists a torodial neighborhood $\mathcal U$ of $M_\mathbf c^0$, $\mathcal U\thickapprox \mathbb T^{r+1}\times D^m$,
endowed with cosymplectic action-angle coordinates
\begin{align*}
&\varphi_1,\dots,\varphi_{r+1},I_1,\dots,I_{r+1},q_1,\dots,q_k,p_1,\dots,p_k \quad (2k=m-r),\\
&I_\mu=I_{\mu}(f_1,\dots,f_r), \quad  q_l=q_l(f_1,\dots,f_m), \quad  p_l=p_l(f_1,\dots,f_m),
\end{align*}
where $\varphi_1,\dots,\varphi_{r+1}$ are angle coordinates of  $\mathbb T^{r+1}$, such that the local expressions for $\omega$ and $\eta$ are of the form
\begin{equation}\label{omegaeta}
\omega=\sum_{\mu=1}^{r+1} d\varphi_\mu\wedge  dI_\mu+\sum_{l=1}^k dq_l\wedge dp_l,
\qquad \eta=dF+\sum_{\mu=1}^{r+1} b_{\mu,r+1}d\varphi_\mu,
\end{equation}
where $F=F(f_1,\dots,f_m)$ and $b_{\mu,r+1}$ are constants.

The Reeb flow \eqref{REEBflow} on $\mathcal U$ is conditionally periodic:
$$
\dot \varphi_\mu=\omega_\mu(I_1,\dots,I_{r+1}), \quad \dot I_\mu=0, \quad \dot q_l=0,\quad  \dot p_l=0,
$$
where the frequencies $\omega_\mu$ are solutions of the linear system
\begin{equation}\label{sistem}
b_{1\mu}\omega_1+\dots+b_{r+1,\mu}\omega_{r+1}=\delta_{r+1,\mu}, \qquad \mu=1,\dots,r+1
\end{equation}
and $b_{\mu\nu}=\dfrac{\partial I_{\mu}}{\partial f_{\nu}}$,  $\mu=1,\dots,r+1$, $\nu=1,\dots,r$.
\end{thm}

Note that the action variables $I_1,\dots,I_{r+1}$ are redundant --- there are $r$ functionally independent functions among them.

\begin{proof}
\emph{Step 1.} Consider a neighborhood $\mathcal U$ of $M_\mathbf c^0$ where
\begin{equation}\label{fibration}
\pi=(f_1,\dots,f_m)\colon \mathcal U \longrightarrow  D^m\subset \R^m(z_1,\dots,z_m)
\end{equation}
is a submersion onto a neighborhood $D^m$ of $\mathbf c=(c_1,\dots,c_m)$, with fibers that are $(r+1)$--dimensional tori.
The functions $f_i$ can be also considered as coordinates of the ball $D^m$.

We look at foliations $\mathcal{P}$ and $\mathcal{Q}$ of $\mathcal U$, defined by the mappings $\pi$ and
\[
\rho=(f_1,\dots,f_r)\colon\mathcal U\longrightarrow D^r\subset \R^r(z_1,\dots,z_r),
\]
respectively.
From the proof of Theorem \ref{TEOREMA},  with $Y_H$ replaced by $Z$, we have that
\[
\langle Z,X_{f_1},\dots,X_{f_r}\rangle\vert_x=T_x\mathcal{P}
\]
and $\mathcal X=\langle Z, X_{f_1},\dots,X_{f_r}\rangle$
is an Abelian Lie algebra of symmetries of the Reeb flow \eqref{REEBflow} that has the leaves of $\mathcal P$ as orbits.

Note that $(T_x\mathcal{P})^\omega=T_x\mathcal{Q}$ and
since $T_x\mathcal{P}\subset T_x\mathcal{Q}$, it follows that $T_x\mathcal{P}$ is isotropic and $T_x\mathcal{Q}$ is coisotropic with respect to $\omega$. Also,
$\omega\vert_\mathcal P=0$.

Indeed, first note that $X_{f_j}(f_i)=\{f_i,f_j\}=0$, for all $i=1,\dots,r$, $j=1,\dots,m$. Also, since $f_1,\dots,f_r$ are integrals of \eqref{REEBflow},
 $Z(f_i)=0$, $i=1,\dots,r$. In other words,
\[
\langle Z,X_{f_1},\dots,X_{f_m}\rangle\vert_x\subset T_x\mathcal{Q}.
\]
In the same way we proved that $Z,X_{f_1},\dots,X_{f_r}$ are independent vector fields, we have that
$Z,X_{f_1},\dots,X_{f_m}$ are independent. Thus, the dimension of
$\langle Z,X_{f_1},\dots,X_{f_m}\rangle\vert_x$ equals $m+1$ and
\[
\langle Z,X_{f_1},\dots,X_{f_m}\rangle\vert_x=T_x\mathcal{Q}.
\]

Further, since $i_Z\omega=0$ and $\{f_i,f_j\}=0$, $i=1,\dots,r$, $j=1,\dots,m$, we have
\begin{align*}
\omega(Z,X_{f_j})=0,\quad \omega(Z,Z)=0, \quad \omega(X_{f_i},X_{f_j})=0, \quad i=1,\dots,r, \quad j=1,\dots,m.
\end{align*}
Thus, $(T_x\mathcal{P})^\omega=T_x\mathcal{Q}$ and $(T_x\mathcal{Q})^\omega=T_x\mathcal{P}$.

\medskip

\emph{Step 2.}
Let $\Sigma$ be an $(n-m)$--dimensional surface which intersects every torus $\pi^{-1}(\mathbf c')$ close to $M_\mathbf c^0=\pi^{-1}(\mathbf c)$ in a unique point
$x(\mathbf c')$.
Every point $y\in \pi^{-1}(\mathbf c')$ can be obtained from the point $x(\mathbf c')$ along the flows of commuting vector fields $X_{f_1},\dots,X_{f_r},Z$ after
times
$t_1,\dots,t_r,t_{r+1}$ and $\pi^{-1}(\mathbf c')\cong \mathbb T^{r+1}=\R^{r+1}/\Gamma_{\mathbf c'}$ ,
where the lattice $\Gamma_{\mathbf c'}$ depends on $\mathbf c'$
\cite{Ar}.

There exist angle coordinates $\theta_1,\dots,\theta_{r+1}$, such that $(\theta_1,\dots,\theta_{r+1},f_1,\dots,f_{m})$ are coordinates on
$\mathcal U$ (referred  as \emph{Liouville coordinates} in \cite{Zung}) and
\begin{equation}\label{VAZNO}
\dfrac{\partial}{\partial\theta_\mu}=\sum\limits_{\nu=1}^{r+1} b_{\mu\nu}X_\nu, \qquad \mu=1,\dots,r+1,
\end{equation}
for some functions $b_{\mu\nu}=b_{\mu\nu}(f_1,\dots,f_m)$.
Here we set $X_{\mu}=X_{f_{\mu}}$, $\mu=1,\dots,r$, $X_{r+1}=Z$.

Recall that $(T_x\mathcal{P})^\omega = T_x\mathcal{Q}$, $T_x\mathcal{P}\subset T_x\mathcal{Q}$ and $\omega\vert_\mathcal P=0$. It follows that
\[
\omega=\sum\limits_{\mu=1}^{r+1}\sum\limits_{\nu=1}^r \widetilde{b}_{\mu\nu}d\theta_\mu\wedge df_\nu+\sum\limits_{1\leqslant i<j\leqslant m} a_{ij}df_i\wedge df_j,
\qquad \eta=\sum\limits_{\mu=1}^{r+1}\widetilde{b}_{\mu,r+1}d\theta_\mu+\sum\limits_{j=1}^m a_jdf_j.
\]

We will prove that $b_{\mu\nu}=\widetilde{b}_{\mu\nu}$, $\mu,\nu=1,\dots,r+1$.

Observe that  from \eqref{definicije} we have relations
 $i_{X_{f_\lambda}}\omega=df_\lambda-Z(f_\lambda)\eta=df_\lambda$ (since $Z(f_\lambda)=0$) and $\eta(X_{f_\lambda})=0$, $\lambda=1,\dots,r$.
  Also, by definition of $Z$,  $i_Z\omega=0$ and $\eta(Z)=1$. Now for every $\lambda=1,\dots,r$ and for every $\xi$ we have
\begin{align*}
df_\lambda(\xi)=i_{X_{f_\lambda}}\omega(\xi)&=\sum\limits_{\mu=1}^{r+1}\sum\limits_{\nu=1}^r \widetilde{b}_{\mu\nu}d\theta_\mu\wedge
df_\nu(X_{f_\lambda},\xi)+\sum\limits_{1\leqslant i<j\leqslant m} a_{ij}df_i\wedge df_j(X_{f_\lambda},\xi)\\
&=\sum\limits_{\mu=1}^{r+1}\sum\limits_{\nu=1}^r \widetilde{b}_{\mu\nu}d\theta_\mu(X_{f_\lambda})df_\nu(\xi).
\end{align*}

Therefore, we conclude that
\begin{equation}\label{B1}
\sum\limits_{\mu=1}^{r+1}\widetilde{b}_{\mu\nu}d\theta_\mu(X_{f_\lambda})=\delta_{\lambda\nu},\qquad  \lambda,\nu=1\dots,r.
\end{equation}

For every $\xi$ we also have
\[
0=i_Z\omega(\xi)=\sum\limits_{\mu=1}^{r+1}\sum\limits_{\nu=1}^r \widetilde{b}_{\mu\nu}d\theta_\mu\wedge df_\nu(Z,\xi)+\sum\limits_{1\leqslant i<j\leqslant m}
a_{ij}df_i\wedge df_j(Z,\xi).
\]

Since $df_i(Z)=Z(f_i)=0$ for all $i=1,\dots,m$, we get $df_i\wedge df_j(Z,\xi)=0$ and
\[
d\theta_\mu\wedge df_\nu(Z,\xi)=\begin{vmatrix}
d\theta_\mu(Z) & df_\nu(Z) \\
d\theta_\mu(\xi) & df_\nu(\xi)
\end{vmatrix}=\begin{vmatrix}
d\theta_\mu(Z) & 0 \\
d\theta_\mu(\xi) & df_\nu(\xi)
\end{vmatrix}=d\theta_\mu(Z)df_\nu(\xi).
\]

Hence,
\begin{align*}
0=\sum\limits_{\mu=1}^{r+1}\sum\limits_{\nu=1}^r \widetilde{b}_{\mu\nu}d\theta_\mu\wedge df_\nu(Z,\xi)+\sum\limits_{1\leqslant i<j\leqslant m} a_{ij}df_i\wedge
df_j(Z,\xi)=\sum\limits_{\mu=1}^{r+1}\sum\limits_{\nu=1}^r \widetilde{b}_{\mu\nu}d\theta_\mu(Z)df_\nu(\xi),
\end{align*}
and we obtain
\begin{equation}\label{B2}
\sum\limits_{\mu=1}^{r+1}\widetilde{b}_{\mu\nu}d\theta_\mu(Z)=0,\qquad  \nu=1,\dots,r.
\end{equation}

Next, from the definition of the Reeb vector field and the Hamiltonian vector fields \eqref{definicije},
and the relations $df_j(Z)=Z(f_j)=0$, $j=1,\dots,m$, we get
\begin{align}
\label{B3} &0=\eta(X_{f_\lambda})=\sum\limits_{\mu=1}^{r+1}\widetilde{b}_{\mu,r+1}d\theta_\mu(X_{f_\lambda})+\sum\limits_{j=1}^m a_j df_j(X_{f_\lambda})=
\sum\limits_{\mu=1}^{r+1}\widetilde{b}_{\mu,r+1}d\theta_\mu(X_{f_\lambda}), \\
\label{B4} &1=\eta(Z)=\sum\limits_{\mu=1}^{r+1}\widetilde{b}_{\mu,r+1}d\theta_\mu(Z)+\sum\limits_{j=1}^m a_j df_j(Z)
=\sum\limits_{\mu=1}^{r+1}\widetilde{b}_{\mu,r+1}d\theta_\mu(Z).
\end{align}

The equations \eqref{B1}, \eqref{B2}, \eqref{B3}, \eqref{B4} in a compact form can be written as
\[
\sum_{\mu=1}^{r+1} X_\lambda(d\theta_\mu) \widetilde b_{\mu\nu}=\delta_{\lambda\nu}, \qquad \lambda,\nu=1,\dots,r+1,
\]
where $X_{\mu}=X_{f_{\mu}}$, $\mu=1,\dots,r$, $X_{r+1}=Z$.

On the other hand, from \eqref{VAZNO} we also have
\[
\delta_{\mu\lambda}=\dfrac{\partial}{\partial\theta_\mu}(d\theta_\lambda)=
\sum\limits_{\nu=1}^{r+1} b_{\mu\nu}X_\nu(d\theta_\lambda) \qquad \mu,\lambda=1,\dots,r+1.
\]

Therefore, as the right inverse and the left inverse of the matrix $\big(X_\lambda(d\theta_\mu)\big)$, the matrices $\big(\widetilde{b}_{\mu\nu}\big)$ and
$\big(b_{\mu\nu}\big)$ are equal.

\medskip

\emph{Step 3.} From the Cartan formula $L_{X_\mu}=i_{X_\mu}\circ d +d\circ i_{X_\mu}$ and the closedness of $\omega$ and $\eta$, it follows that
\[
L_{X_{\mu}}(\omega)=0, \qquad L_{X_{\mu}}(\eta)=0, \qquad \mu=1,\dots,r+1,
\]
i.e., $\omega$ and $\eta$ are invariant with respect to the action of $\mathcal X$.
We have the natural induced action of $\mathbb T^{r+1}$ on $\mathcal U$ by the translations in
angle coordinates (\emph{Liouville torus action} \cite{Zung}).
According to Theorem 2.2 (Fundamental conservation property of Liouville torus actions) in Zung \cite{Zung},
we get that $\omega$ and $\eta$ are invariant with respect
to the torus action as well:
\[
L_{V_{\mu}}(\omega)=0, \qquad L_{V_{\mu}}(\eta)=0, \quad V_\mu=\dfrac{\partial}{\partial\theta_\mu}, \quad \mu=1,\dots,r+1.
\]

We can write $\omega=\alpha+\beta$, $\eta=\widetilde\alpha+\widetilde\beta$, where
\begin{align*}
&\beta=\displaystyle\sum\limits_{\mu=1}^{r+1}\sum\limits_{\nu=1}^r b_{\mu\nu}d\theta_\mu\wedge df_\nu, \quad
\alpha=\displaystyle\sum\limits_{1\leqslant i<j\leqslant m} a_{ij}df_i\wedge df_j, \\
&\widetilde\beta=\sum\limits_{\mu=1}^{r+1}{b}_{\mu,r+1}d\theta_\mu, \qquad\qquad \widetilde\alpha=\sum\limits_{j=1}^m a_j df_j.
\end{align*}

It is obvious that $L_{V_\mu}(\beta)=0$ and $L_{V_\mu}(\widetilde\beta)=0$, $\mu=1,\dots,r+1$. Therefore,
$L_{V_\mu}(\alpha)=0$ and $L_{V_\mu}(\widetilde\alpha)=0$, and the functions $a_{ij}$, $a_j$ depend only on $(f_1,\dots,f_m)$.

Since $a_j$ and $b_{\mu,r+1}$ do not depend on $\theta_1,\dots,\theta_{r+1}$ and $d\eta=0$, we get that
$\widetilde\alpha=\pi^*\widetilde\alpha'$, where $\widetilde\alpha'$ is a closed form on $D^m$ and that $b_{\mu,r+1}$ are constants.
According to the Poincar\'e lemma \cite{LM},
there exists a function  $\widetilde{F}(z_1,\dots,z_m)$ on $D^m$ (if necessary, by taking a smaller ball $D^m$),
such that $\widetilde\alpha'=d\widetilde{F}$. Then $\widetilde\alpha=\pi^*d\widetilde{F}=d\pi^*\widetilde{F}=d\widetilde{F}(f_1,\dots,f_m)$ and
\begin{equation}\label{eta}
\eta=d\widetilde{F}+\sum\limits_{\mu=1}^{r+1} b_{\mu,r+1}d\theta_\mu.
\end{equation}

Next, as $d\omega=d\alpha+d\beta=0$ and $b_{\mu\nu}$ and $a_{ij}$ depend only on $(f_1,\dots,f_m)$,
it follows that
\[
d\alpha=0, \qquad d\beta=0.
\]

Introduce 1--forms $\beta_\mu=\displaystyle\sum\limits_{\nu=1}^r b_{\mu\nu}df_\nu$, $\mu=1,\dots,r+1$. Then
$\beta=\displaystyle\sum\limits_{\mu=1}^{r+1}d\theta_\mu\wedge\beta_\mu$ and
\[
d\beta=-\sum\limits_{\mu=1}^{r+1} d\theta_\mu\wedge d\beta_\mu=0 \quad \Rightarrow \quad d\beta_\mu=0.
\]

Thus, as in the case of $\widetilde\alpha$, eventually by shrinking of $D^m$, we have that $\beta_\mu=\pi^*\beta_\mu'$, where
$\beta_\mu'=dI_\mu(z_1,\dots,z_m)$, and  $\beta_\mu=dI_\mu(f_1,\dots,f_m)$, $\mu=1,\dots,r+1$.
Since $dI_\mu=\beta_\mu$,  the functions $I_\mu$ depend only on $f_1,\dots,f_r$
 and
\[
b_{\mu\nu}=\dfrac{\partial I_{\mu}}{\partial f_{\nu}},  \quad \mu=1,\dots,r+1, \quad \nu=1,\dots,r.
\]

As a result,
\[
\omega=\sum_{\mu=1}^{r+1} d\theta_\mu \wedge dI_\mu+\displaystyle\sum\limits_{1\leqslant i<j\leqslant m} a_{ij}df_i\wedge df_j.
\]
Also, since the $((r+1)\times (r+1))$--matrix $\big(b_{\mu\nu}\big)$ is non-singular,
its submatrix $\big(\partial I_\mu/\partial f_\nu\big)$ has rank $r$ and
\[
\langle dI_1,\dots, dI_{r+1}\rangle=\langle df_1,\dots,df_r\rangle.
\]
We can conclude that every 1--form that can be written in terms of $df_1,\dots,df_r$ can also be written (in a non-unique way) in terms of $dI_1,\dots,dI_{r+1}$.

\begin{rem}
Since $\omega\vert_{\pi^{-1}(\mathbf c')}=0$ and the base $D^m$ of the fibration \eqref{fibration} is diffeomorphic  to a $m$--dimensional ball, the class $[\omega]$
in the second cohomology group $H^2(\mathcal U,\mathbb R)$ is equal to zero.
Therefore, there exist 1--form $\lambda$, such that $\omega=-d\lambda$ on $\mathcal U$. As in the symplectic case,
we can  also define redundant action variables
 by
\[
I_\mu=\frac{1}{2\pi}\int_{\gamma_\mu} \lambda, \qquad \mu=1,\dots,r+1,
\]
where $\gamma_1,\dots,\gamma_{r+1}$ form a basis of 1--cycles on $\pi^{-1}(\mathbf c')$.
The functions $I_\mu$ do not depend on the angle variables $\theta_\nu$. Indeed, by the Stokes theorem
\[
\int_{\gamma_\mu} \lambda-\int_{\gamma'_\mu}\lambda=\int_{\sigma} d\lambda=\int_\sigma\omega=0,
\]
where $\sigma\subset\pi^{-1}(\mathbf c')$, $\partial\sigma=\gamma_\mu-\gamma_\mu'$.  Thus, $I_\mu=I_\mu(f_1,\dots,f_{m})$.
Now, by using the identity
$\lambda=\sum\limits_\mu I_\mu d\theta_\mu+\sum\limits_i c_i df_i$,  one can give a different proof for steps 2 and 3.
\end{rem}

\emph{Step 4.} We need the following statement, usually formulated for symplectic linear spaces \cite{LM}:

\begin{lem}\label{POMOC}
Let $\Omega$ be a 2--form on a vector space $\mathbb V$.  Let $\mathbb Q<\mathbb V$ be a coisotropic subspace with respect
to $\Omega$: $\mathbb P=\mathbb Q^\Omega < \mathbb Q$. Further, let $\mathbb S<\mathbb Q$ be the subspace transverse to $\mathbb P$
($\mathbb Q=\mathbb P\oplus \mathbb S$).
Then the restriction of $\Omega$ to $\mathbb S$ is non-degenerate. In other words, the form $\Omega$ induces a symplectic form on
the quotient space $\mathbb Q/\mathbb P$.
\end{lem}

Let $x_0\in M_\mathbf c^0$.
We apply Lemma \ref{POMOC} for $\mathbb V=T_{x_0}\mathcal{U}$, the form $\Omega=\omega\vert_{x_0}$,
and the spaces
\begin{align*}
&\mathbb Q=T_{x_0} \mathcal Q=\langle Z,X_{f_1},\dots,X_{f_m}\rangle\vert_{x_0}=
\left\langle\frac{\partial}{\partial\theta_1},\dots,\frac{\partial}{\partial\theta_{r+1}},\frac{\partial}{\partial f_{r+1}},\dots,\frac{\partial}{\partial
f_{m}}\right\rangle\Big\vert_{x_0}, \\
&\mathbb P=T_{x_0}\mathcal P=\langle Z,X_{f_1},\dots,X_{f_r}\rangle\vert_{x_0}=
\left\langle \frac{\partial}{\partial\theta_1},\dots,\frac{\partial}{\partial\theta_{r+1}}\right\rangle\Big\vert_{x_0},\\
&\mathbb S=\left\langle\frac{\partial}{\partial f_{r+1}},\dots,\frac{\partial}{\partial f_{m}}\right\rangle\Big\vert_{x_0}.
\end{align*}

Since $\omega\vert_\mathbb S=\Omega\vert_\mathbb S$ is non-degenerate, there exists a linear change of variables on $\mathcal U$
\[
f_{i}=f_i(Q_1,\dots,Q_k,P_1,\dots,P_k), \quad i=r+1,\dots,m,
\]
such that $\omega$ at $x_0$ gets the form
\begin{align*}
\omega\vert_{x_0}=&\sum_{\mu=1}^{r+1}d\theta_\mu\wedge dI_\mu+\sum_{i,j=1}^r B_{ij}df_i\wedge df_j\\
&+\sum_{i=1}^r\sum_{l=1}^k(E_{il} df_i\wedge dQ_l+F_{il}df_i\wedge dP_l)+\sum_{l=1}^k dQ_l\wedge dP_l,
\end{align*}
where $B_{ij}=-B_{ji}$, $i,j=1,\dots,m$, and $E_{il}$, $F_{il}$, $i=1,\dots,r$, $l=1,\dots,k$, are some constants.
Next, we can perform an additional linear change of variables
\[
Q_{l}=q_l-\sum_{i=1}^r F_{il} f_i, \qquad P_{l}=p_l+\sum_{i=1}^r E_{il} f_i, \quad l=1,\dots,k,
\]
implying that the original 2--form gets the form
\[
\omega=\sum_{\mu=1}^{r+1}d\theta_\mu\wedge dI_\mu+(\text{2--form $\alpha$ written in terms of $f_1,\dots,f_r,q_1,\dots,q_k,p_1,\dots,p_k$}),
\]
such that
\[
\omega\vert_{x_0}=\sum_{\mu=1}^{r+1}d\theta_\mu\wedge dI_\mu+\Delta
+\sum_{l=1}^k dq_l\wedge dp_l
\]
where
\begin{align*}
\Delta=      \sum\limits_{i,j=1}^r B_{ij}df_i\wedge df_j+\sum_{1\le i<j\le r}\sum_{l=1}^k( F_{il}E_{jl}-E_{il}F_{jl})df_i\wedge df_j.
\end{align*}

Now we can apply a cosymplectic variant of Moser's trick  for $(\mathcal U,\omega,\eta)$ and $(\mathcal U,\omega_0,\eta)$, where
\[
\omega_0=\sum_{\mu=1}^{r+1}d\theta_\mu\wedge dI_\mu+\Delta
+\sum_{l=1}^k dq_l\wedge dp_l
\]
is a closed 2--form on $\mathcal U$ that coincides with $\omega$ at the torus $M^0_\mathbf c$.

From now on, for simplicity, we use the identification $\mathcal U\thickapprox \mathbb T^{r+1}\times D^m$
and consider $(\theta_1,\dots,\theta_{r+1})$ and $(f_1,\dots,f_r,q_1,\dots,q_k,p_1,\dots,p_k)$ as coordinates on $\mathbb T^{r+1}$ and $D^m$, respectively.
Also, without loss of generality, we assume that the torus $M_\mathbf c^0$ is given as the zero level-set
\[
M_\mathbf c^0=\{(\theta_1,\dots,\theta_{r+1},0,\dots,0)\,\vert\, \theta\in [0,2\pi]\}.
\]

Recall that an isotopy on a differential manifold $M$ is a smooth map $\phi:M\times I\longrightarrow M$ such that the family of maps
$\phi_t=\phi(\cdot,t):M\longrightarrow M$ is a family of diffeomorphisms with $\phi_0=\Id$. For each isotopy $\phi_t$ one can construct a time-dependent vector field
$Y_t$ such that
\begin{equation}\label{izotopija}
\dfrac{d\phi_t}{dt}(p)=Y_t(\phi_t(p)).
\end{equation}
Then it is well known that
\begin{equation}\label{LijevIzvod2}
\dfrac{d}{dt}\phi_t^*\alpha_t=\phi_t^*\big(L_{Y_t}\alpha_t+\dfrac{d\alpha_t}{dt}\big),
\end{equation}
where $\alpha_t$ a smooth family of $k$--forms on $M$.

Let $I$ be an open interval in $\R$ such that $[0,1]\subset I$. Define a family of closed 2--forms
\begin{equation}\label{omega_t}
\omega_t=\omega_0+
t(\omega-\omega_0)=\omega_0+t(\alpha-\Delta-\sum_{l=1}^k dq_l\wedge dp_l), \qquad t\in I.
\end{equation}

From the Poincar\'e lemma, there exists a 1--form $\tau$ on $D^m$ (if necessary, by shrinking of $D^m$), such that
$\omega_0-\omega=d\tau$, $\tau\vert_{M^0_\mathbf c}=0$.

The idea of Moser's trick is to find an isotopy $\phi_t$ such that
\begin{align}
&\phi_t^*\omega_t=\omega_0,\label{Mozer1}\\
&\phi_t^*\eta=\eta,\label{Mozer2}
\end{align}
for all $t\in[0,1]$. By \eqref{LijevIzvod2} and the Cartan formula applied on \eqref{Mozer1}, \eqref{Mozer2}, we get
\begin{align*}
&0=\dfrac{d}{dt}\phi_t^*\omega_t=\phi_t^*\big(L_{Y_t}\omega_t+\dfrac{d\omega_t}{dt}\big)=
\phi_t^*(L_{Y_t}\omega_t+\omega-\omega_0)=\phi_t^*(d(i_{Y_t}\omega_t)+\omega-\omega_0),\\
&0=\dfrac{d}{dt}\phi_t^*\eta=\phi_t^* L_{Y_t}\eta=\phi_t^*(d(\eta(Y_t))).
\end{align*}

Vice versa, if $Y_t$ is a vector field that satisfies equations
\begin{align*}
d(i_{Y_t}\omega_t)=\omega_0-\omega=d\tau, \qquad \eta(Y_t)=\mathrm{const},
\end{align*}
then the associated isotopy $\phi_t$ leads to \eqref{Mozer1}, \eqref{Mozer2}.

We have a family of cosymplectic structures $(\omega_t,\eta)$ on $\mathcal U$, at least by contraction of $D^m$.
Thus, the equations
\begin{align}\label{Yt}
i_{Y_t}\omega_t=\tau, \qquad \eta(Y_t)=0,
\end{align}
defines the family of vector fields $Y_t$,
$Y_t\vert_{M_{\mathbf c}^0}=0$.

From \eqref{omega_t} and the fact that $\tau$ is a 1--form on $D^m$, we get that
the vector fields $Y_t$, $t\in I$ do not have terms with $\dfrac{\partial}{\partial f_\nu}$,  $\nu=1,\dots,r$.
Next, since $\omega_t$, $\tau$, and $\eta$ are invariant with respect to the $\mathbb T^{r+1}$--action,
the coefficients of the vector field $Y_t$ do not depend on $\theta_1,\dots,\theta_{r+1}$ as well.
Since the torus $M_{\mathbf c}^0$ is compact, eventually by taking a smaller neighborhood $\mathcal U$, we get the isotopy $\phi_t\colon\mathcal
U\to\mathcal U$,
\begin{align*}
&\theta_\mu^t=\theta_\mu^t(\theta_1,\dots,\theta_{r+1},f_1,\dots,f_r,q_1,\dots,q_k,p_1,\dots,p_k),&& \mu=1,\dots,r+1,\\
&f_\nu^t=f_\nu^t(\theta_1,\dots,\theta_{r+1},f_1,\dots,f_r,q_1,\dots,q_k,p_1,\dots,p_k),&& \nu=1,\dots,r,\\
&q_l^t=q_l^t(\theta_1,\dots,\theta_{r+1},f_1,\dots,f_r,q_1,\dots,q_k,p_1,\dots,p_k),&& l=1,\dots,k,\\
&p_l^t=p_l^t(\theta_1,\dots,\theta_{r+1},f_1,\dots,f_r,q_1,\dots,q_k,p_1,\dots,p_k),&& l=1,\dots,k,
\end{align*}
which satisfies \eqref{Mozer1}, \eqref{Mozer2}  and $\phi_t\vert_{M^0_\mathbf c}=\Id$.
We have (see \eqref{izotopija}):
\[
Y_t\vert_{(\theta_1^t,\dots,\theta_{r+1}^t,f_1^t,\dots,f_r^t,q_1^t,\dots,q_k^t,p_1^t,\dots,p_k^t)}=\sum\limits_{\mu=1}^{r+1}
\dfrac{d\theta_\mu^t}{dt}\dfrac{\partial}{\partial\theta_\mu}+\sum\limits_{\nu=1}^r \dfrac{df_\nu^t}{dt}\dfrac{\partial}{\partial f_\nu}+\sum\limits_{l=1}^k
\dfrac{dq_l^t}{dt}\dfrac{\partial}{\partial q_l}+\sum\limits_{l=1}^k \dfrac{dp_l^t}{dt}\dfrac{\partial}{\partial p_l}.
\]
It follows that $\dfrac{df_\nu^t}{dt}=0$ and $\dfrac{d\theta_\mu^t}{dt}$, $\dfrac{dq_l^t}{dt}$, $\dfrac{dp_l^t}{dt}$ do not depend on $\theta_1,\dots,\theta_{r+1}$.

As a result,
we obtain a diffeomorphism $\Phi=\phi_1\colon \mathcal U\to \mathcal U$ of the form
\begin{align*}
&\theta_\mu^1=\theta_\mu+D_\mu(f_1,\dots,f_r,q_1,\dots,q_k,p_1,\dots,p_k),\,\quad \mu=1,\dots,r+1,\\
&q_l^1=q_l^1(f_1,\dots,f_r,q_1,\dots,q_k,p_1,\dots,p_k), \qquad\qquad  l=1,\dots,k,\\
&p_l^1=p_l^1(f_1,\dots,f_r,q_1,\dots,q_k,p_1,\dots,p_k), \qquad\qquad  l=1,\dots,k,\\
& D_\mu(0,\dots,0)=0, \qquad q_l^1(0,\dots,0)=0, \qquad p_l^1(0,\dots,0)=0,
\end{align*}
such that
$\Phi^*\omega=\omega_0$, $\Phi^*\eta=\eta$.

\medskip

\emph{Step 5.}
Since every 1--form that can be expressed in terms of $df_1,\dots,df_r$ can also be expressed in terms of $dI_1,\dots,dI_{r+1}$ and
\begin{align*}
\Delta=d\big(\sum\limits_{j=1}^r \big(\sum\limits_{i=1}^r B_{ij}f_i\big)df_j+\sum_{l=1}^k (F_{l1}f_1+\dots+F_{lr}f_r)(E_{l1}df_1+\dots+E_{lr}df_r)\big),
\end{align*}
we can present (non-uniquely) $\Delta$ in the form
\[
\Delta=d\big(\sum\limits_{\mu=1}^{r+1} G_\mu(f_1,\dots,f_r)dI_\mu\big)=\sum_{\mu=1}^{r+1} dG_\mu \wedge dI_\mu,
\]
 for some functions $G_\mu$, $\mu=1,\dots,r+1$.

 Finally, after the translation of angle variables
\begin{equation}\label{translacija}
\theta_\mu=\varphi_\mu-G_\mu,
\end{equation}
we obtain the required 2--form:
\begin{align*}
\omega_0=\sum_{\mu=1}^{r+1}d\theta_\mu\wedge dI_\mu+\Delta+\sum_{l=1}^k dq_l\wedge dp_l=\sum_{\mu=1}^{r+1}d\varphi_\mu\wedge dI_\mu+\sum_{l=1}^k dq_l\wedge dp_l.
\end{align*}

It is obvious that the 1--form \eqref{eta}, after the above translations in angle variables  takes the form
\[
\eta=dF+\sum\limits_{\mu=1}^{r+1} b_{\mu,r+1}d\varphi_\mu,
\]
for some function $F=F(f_1,\dots,f_m)$.

It remains to observe that the frequencies $\omega_1,\dots,\omega_{r+1}$ of the Reeb flow are the components of the vector filed $X_{r+1}=Z$
with respect to the basis $V_\mu=\dfrac{\partial}{\partial \varphi_\mu}=\dfrac{\partial}{\partial \theta_\mu}$, $\mu=1,\dots,r+1$.
Since the vector fields  $X_1,\dots,X_{r+1}$ and $V_1,\dots, V_{r+1}$ are related by \eqref{VAZNO}, we obtain that the frequencies $\omega_\nu$
are solutions of the linear system \eqref{sistem}.
\end{proof}

\subsection{Integrability of Hamiltonian vector fields}\label{sec3.2}

For  $m=r=n$ (commutative case), there exists a torodial neighborhood $\mathcal U$ of $M_\mathbf c^0$,
$\mathcal U\thickapprox\mathbb T^{n+1}\times D^n$,
endowed with angle coordinates
$\varphi_1,\dots,\varphi_{n+1}$ of the torus $\mathbb T^{n+1}$ and redundant action coordinates $I_1,\dots,I_{n+1}$ of the ball $D^n$, such that the local
expressions for $\omega$ and $\eta$ are of the form
\begin{equation}\label{oe}
\omega=\sum_{i=1}^{n+1} d\varphi_i\wedge  dI_i, \qquad \eta=dF+\sum\limits_{i=1}^n b_{i,n+1}d\varphi_i,
\end{equation}
where action variables $I_i$ and $F$ depend only on values of the integrals $f_1,\dots,f_n$ and $b_{i,n+1}$ are constants.
\footnote{One can compare \eqref{oe} and the universal cosymplectic structure given in \cite{LT}.}

The Reeb flow on $\mathcal U$ is conditionally periodic:
\[
\dot \varphi_i=\omega_i(I_1,\dots,I_{n+1}), \qquad \dot I_i=0, \qquad i=1,\dots,n+1,
\]
where the frequencies $\omega_i$ are solutions of the linear system
\begin{equation}\label{bZ}
b_{1i}\,\omega_1+\dots+b_{n+1,i}\,\omega_{n+1}=\delta_{n+1,i}, \qquad i=1,\dots,n+1,
\end{equation}
and $b_{ij}=\dfrac{\partial I_i}{\partial f_j}$,  $i=1,\dots,n+1$, $j=1,\dots,n$.
The flows $X_{f_1},\dots,X_{f_n}$ are simultaneously solvable by quadratures and quasi-periodic.
The Hamiltonian flows of  $f_1,\dots,f_n$ are of the form
$$
X_{f_k}=\Omega^k_1\frac{\partial}{\partial\varphi_1}+\dots+\Omega^k_{n+1}\frac{\partial}{\partial\varphi_{n+1}},
$$
with frequencies $\Omega^k_i$ determined by the conditions
$$
b_{1j}\,\Omega^k_1+\dots+b_{n+1,j}\,\Omega^k_{n+1}=\delta_{kj}, \qquad j=1,\dots,n+1.
$$

Note that the Hamiltonian flows $X_{f_k}$, in general, are everywhere dense over
$(n+1)$--dimensional tori $\pi^{-1}(\mathbf c')$, and, therefore are not completely integrable in the sense of the usual complete integrability of Hamiltonian flows
on Poisson manifolds.

On the other hand, if there exists a globally defined Casimir function $f_0$ of the Poisson structure, then the Hamiltonian flows $X_{f_k}$
are completely integrable by means of integrals $f_0,f_1,\dots,f_n$. A generic motion is quasi-periodic over $n$ dimensional Lagrangian tori and the construction of
the corresponding action-angle coordinates can be found, e.g., in \cite{LMV}.

\subsection*{Acknowledgments}
B.J. was supported by the Project  7744592 MEGIC ”Integrability
and Extremal Problems in Mechanics, Geometry and Combinatorics” of the Science Fund of Serbia. K.L. was partially funded by the Ministry of Education, Science
and Technological Developments of the Republic of Serbia: grant number 451-03-68/2022-14/200104 with Faculty of Mathematics.

\end{document}